\documentclass[preprint,12pt]{elsarticle}

\usepackage{amssymb,amsmath,amsthm}
\usepackage{amsfonts}
\usepackage{etoolbox}

\makeatletter
\def\@author#1{\g@addto@macro\elsauthors{\normalsize%
    \def\baselinestretch{1}%
    \upshape\authorsep#1\unskip\textsuperscript{%
      \ifx\@fnmark\@empty\else\unskip\sep\@fnmark\let\sep=,\fi
      \ifx\@corref\@empty\else\unskip\sep\@corref\let\sep=,\fi
      }%
    \def\authorsep{\unskip,\space}%
    \global\let\@fnmark\@empty
    \global\let\@corref\@empty  %% Added
    \global\let\sep\@empty}%
    \@eadauthor={#1}
}
\patchcmd{\ps@pprintTitle}{\footnotesize\itshape
       Preprint submitted to \ifx\@journal\@empty Elsevier
       \else\@journal\fi\hfill\today}{\relax}{}{}
\makeatother

\theoremstyle{plain}

\newtheorem{theorem}{Theorem}

\newtheorem{lemma}[theorem]{Lemma}
\newtheorem{proposition}[theorem]{Proposition}
\newtheorem{conjecture}[theorem]{Conjecture}
\newtheorem{remark}[theorem]{Remark}

\newproof{pop1}{Proof of Proposition \ref{diamonds}}
\newproof{pop2}{Proof of Proposition \ref{lemseidel}}
\newproof{pop3}{Proof of Proposition \ref{main2}}

\begin{document}    
    \title{An exact extremal result for tournaments and 4-uniform hypergraphs}
    
    \author[A]{Wiam Belkouche}
    \ead{belkouche.wiam@gmail.com}
    
    \author[A]{Abderrahim Boussa\"{\i}ri\corref{cor1}}  
    \cortext[cor1]{Corresponding author}
    \ead{aboussairi@hotmail.com}
    
    \author[A]{Soufiane Lakhlifi}
    \ead{s.lakhlifi1@gmail.com}
    
    \author[A]{Mohammed Zaidi}
    \ead{zaidi.fsac@gmail.com}

%\cortext[cor2]{Principal corresponding author}

\address[A]{Facult\'e des Sciences A\"in Chock,  D\'epartement de Math\'ematiques et Informatique, Laboratoire de Topologie, Alg\`ebre, G\'eom\'etrie et Math\'ematiques Discr\`etes, Universit\'e Hassan II

Km 8 route d'El Jadida,
BP 5366 Maarif, Casablanca, Maroc}
        \begin{frontmatter}
        \begin{abstract}
                            In this paper, we address the following problem due to Frankl and F\"uredi (1984). What is the maximum number of hyperedges in an $r$-uniform hypergraph
            with $n$ vertices, such that every set of $r+1$ vertices contains $0$ or exactly
            $2$ hyperedges? They solved this problem for $r=3$. For $r=4$, a
            partial solution is given by Gunderson and Semeraro (2017) when $n=q+1$ for
            some prime power number $q\equiv3\pmod{4}  $. Assuming the existence
            of skew-symmetric conference matrices for every order divisible by $4$, we
            give a solution for $n\equiv0\pmod{4}  $ and for $n\equiv3\pmod{4}  $.
        \end{abstract}

    \end{frontmatter}
\section{Introduction}
One of the most important problems in extremal combinatorics is to determine
the largest or the smallest possible number of copies of a given object in a
finite combinatorial structure. In the first part of this work, we address
this problem in the case of tournaments. Throughout this paper, we mean by an 
\emph{$n$-tournament},  a tournament with $n$ vertices. It is easy to see that, up to
isomorphism, there are four distinct $4$-tournaments. The two that contain a
single $3$-cycle are called \emph{diamonds}. They consist of a vertex
dominating or dominated by a $3$-cycle. The class of tournaments without
diamonds was characterized by Moon \cite{moon79}. These tournaments appear in
the literature under the names \emph{local orders} \cite{cameron78},\emph{
locally transitive tournaments} \cite{lachlan84} or \emph{vortex-free
tournaments} \cite{Knuth92}. Curiously, there is little work on the number
$\delta_{T}$ of diamonds in an $n-$tournament. To our knowledge, the only
papers dealing with this problem are those of Bouchaala \cite{bouchaala2004}
and Bouchaala et al. \cite{bouchaala13}.  Recently, Bondy \cite{bondy08} asked
for the maximum number of diamonds in an $n$-tournament. To attack this
problem, we will use a relation between the number $\delta_{T}$\ of diamonds
in a tournament $T$ and the coefficients of its Seidel adjacency matrix.
Recall that the \emph{adjacency matrix} of a tournament $T$ with $n$ vertices
$v_{1},\ldots,v_{n}$ is the $n\times n$ matrix $A=(a_{i,j})$ in which
$a_{i,j}$ is $1$ if $v_{i}$ dominates $v_{j}$ and $0$ otherwise. The
\emph{Seidel adjacency matrix} of a tournament $T$ is $S=A-A^{T}$ where
$A^{T}$ is the transpose of $A$. Our first result is stated as follows.

\begin{theorem}
\label{main1}Let $T$ be an $n$-tournament and $S$ its Seidel
adjacency matrix. Then we have the following:

\begin{enumerate}
\item If $n$ is even, then $\delta_{T}\leq\frac{1}{96}n^{2}(n-1)(n-2)$.
Moreover, equality holds if and only if $n\equiv0\pmod{4}$ and $S$ is a
skew-conference matrix.

\item If $n$ is odd, then $\delta_{T}\leq\frac{1}{96}n\left(  n-1\right)
\left(  n-3\right)  (n+1)$. Moreover, equality holds if and only if
$n\equiv3\pmod{4}$ and $S$ is obtained by  deleting a row and the corresponding
column from a skew-conference matrix.
\end{enumerate}
\end{theorem}

Bondy's problem is related to the following particular case of a
problem raised by Frankl and F\"{u}redi \cite{frankl84}. What is the maximum
number of hyperedges of an $4$-uniform hypergraph in which every $5$ vertices
contains $0$ or exactly $2$ hyperedges? We call such hypergraphs, $FF_{4}%
$\emph{-hypergraph}. To construct $FF_{4}$-hypergraphs, we can use Baber's
construction \cite{baber2015}. Baber associates with each tournament
$T=(V,A)$, the $4$-uniform hypergraph $\mathcal{H}_{T}$ on $V$ whose
hyperedges correspond to subsets of $V$ which induce a diamond in $T$. This
hypergraph is an $FF_{4}$-hypergraph because every $5$-tournament contains $0$
or $2$ diamonds. Gunderson and Semeraro \cite{semeraro17} showed that an $FF_{4}%
$-hypergraph with $n$ vertices has at most $\frac{1}{96}n^{2}(n-1)(n-2)$
hyperedges. Moreover, they proved that this bound is reached if $n=q+1$ for
some prime power number $q\equiv3\pmod{4}  $. To prove this, they
considered the $FF_{4}$-hypergraph $\mathcal{H}_{T^{\ast}(q)}$ where $T^{\ast
}(q)$ is the tournament obtained form the Paley tournament $T(q)$ on $q$
vertices by adding a new vertex which dominates every vertex of $T(q)$. The
fact that the $FF_{4}$-hypergraph\emph{ }$\mathcal{H}_{T^{\ast}(q)}$ has
exactly $\frac{1}{96}n^{2}(n-1)(n-2)$ hyperedges follows from Theorem 11 and
19 of \cite{semeraro17}. The Seidel adjacency matrix of $T^{\ast}(q)$ is a
skew-conference matrix. This is obtained via Paley construction \cite{paley33}.

The second part of this paper is devoted to Frankl-F\"{u}redi problem. As we
have mentioned above, a partial solution of this problem is obtained in
\cite{semeraro17}. More precisely, the following result is a particular 
case of \cite[Proposition~14]{semeraro17}.

\begin{proposition}
\label{sem1}The number $e(\mathcal{H})$ of hyperedges in an $FF_{4}%
$-hypergraph $\mathcal{H}$ with $n$ vertices is at most $\frac{1}{96}%
n^{2}(n-1)(n-2)$. Moreover, if $n\equiv0\pmod{4}$, the equality holds if
and only if $\mathcal{H}$ is such that every set of $3$ vertices occurs in
exactly $\frac{n}{4}$ hyperedges.
\end{proposition}

For $n\equiv3\pmod{4}$, we obtain the following.

\begin{proposition}
\label{main2} The number $e(\mathcal{H})$ of hyperedges in an $FF_{4}%
$-hypergraph $\mathcal{H}$ with $n=4t+3$ vertices is at most $\frac{1}%
{96}n\left(  n-1\right)  \left(  n-3\right)  (n+1)$.
\end{proposition}

In the last section, we discuss the existence of $FF_{4}$-hypergraphs for which
the equality holds in Propositions \ref{main1} and \ref{main2}.

\section{Maximum number of diamonds in tournaments}

In this section, we will prove our first main result. We start with the
following lemma, which gives a relation between the number of diamonds in an
$n$-tournament and the sum of principal minors of order $4$ of its associated
Seidel adjacency matrix.

\begin{lemma}
\label{diam principa minors}Let $T$ be an $n$-tournament with $n$ vertices and
$S$ its Seidel adjacency matrix. Then the sum of all $4\times4$ principal
minors of $S$ is equal to $8\cdot\delta_{T}+\binom{n}{4}$.
\end{lemma}

\begin{proof}
We can check that the determinant of the Seidel adjacency matrix of a
$4$-tournament is $9$ if it is a diamond and $1$ otherwise. Moreover, the
number of all $4\times4$ principal minors of $S$ is $\binom{n}{4}$. It follows
that the sum of all $4\times4$ principal minors of $S$ is equal to
$9\cdot\delta_{T}+(\binom{n}{4}-\delta_{T})=$ $8\cdot\delta_{T}+\binom{n}{4}$.
\end{proof}

Let $M$ be an $n\times n$ complex matrix and $P_{M}(x)=\det(xI-M)=x^{n}%
+\sigma_{1}x^{n-1}+\sigma_{2}x^{n-2}+\cdots+\sigma_{n-1}x+\sigma_{n}$ its
characteristic polynomial, then%

\begin{equation}
\sigma_{k}=(-1)^{k}\sum(\text{all }k\times k\text{ principal minors})\text{ }
\label{eq01}%
\end{equation}

Consider now the $n$ complex eigenvalues $\alpha_{1},\alpha_{2},\ldots
,\alpha_{n}$ of $M$, and denote by $s_{k}$ the $k^{th}$ symmetric function of
the eigenvalues $\alpha_{1},\alpha_{2},\ldots,\alpha_{n}$ of $M$, that is $s_k=\underset{1\leq i_{1}<i_{2}<\cdots<i_{k}\leq n}{\sum}\alpha_{i_{1}%
}\alpha_{i_{2}}\ldots \alpha_{i_{k}}$. Then, we have
\[
s_{k}=(-1)^{k}\sigma_{k}%
\]

In particular :

\begin{enumerate}
\item $s_{1}=trace(M)=\alpha_{1}+\alpha_{2}+\cdots+\alpha_{n}=-\sigma_{1}$;

\item $s_{n}=\det(M)=\alpha_{1}\alpha_{2}\ldots\alpha_{n}=(-1)^{n}\sigma_{n}$.
\end{enumerate}

When $M$ is a real skew-symmetric matrix, its nonzero eigenvalues are all
purely imaginary and come in complex conjugate pairs $\pm i\lambda_{1}%
,\ldots,\pm i\lambda_{k}$, where $\lambda_{1},\ldots,\lambda_{k}$ are real
positive numbers. Then the characteristic polynomial of $M$ has the form%

\[
P_{M}(x)=x^{l}(x^{2}+\lambda_{1}^{2})(x^{2}+\lambda_{2}^{2})\cdots
(x^{2}+\lambda_{k}^{2})
\]

where $l+2k=n$.

Assume now that $M$ is skew-symmetric and all its off-diagonal entries are
from the set $\{-1,1\}$. Such matrix is sometimes known as a \emph{
skew-symmetric Seidel matrix}. By using \cite[Proposition~1]{mccarthy96},
$\det(S)=0$ if and only if $n$ is odd. Then, if $n$ is even, $l=0$ and
\[
P_{M}(x)=(x^{2}+\lambda_{1}^{2})(x^{2}+\lambda_{2}^{2})\cdots(x^{2}%
+\lambda_{n/2}^{2})
\]
If $n$ is odd, then by using \cite[Proposition~1]{mccarthy96} again, any
$(n-1)\times(n-1)$-principal minor is nonzero and thus, the multiplicity of
the eigenvalue $0$ is $1$. It follows that%

\[
P_{M}(x)=x(x^{2}+\lambda_{1}^{2})(x^{2}+\lambda_{2}^{2})\cdots(x^{2}%
+\lambda_{(n-1)/2}^{2})
\]

To prove Theorem \ref{main1}, we need the following proposition.

\begin{proposition}
\label{diamonds} Let $S$ be a skew-symmetric Seidel matrix and let
$P_{S}(x)=x^{n}+\sigma_{2}x^{n-2}+\sigma_{4}x^{n-4}+\cdots+\sigma_{n-2}%
x^{2}+\sigma_{n}$ its characteristic polynomial. Then we have the following assertions:

\begin{enumerate}
\item $\sigma_{2}=\frac{n(n-1)}{2}$;

\item If $n$ is even, then $\sigma_{4}\leq\frac{1}{8}n(n-1)^{2}(n-2)$, with
equality if and only if $P_{S}(x)=(x^{2}+(n-1))^{\frac{n}{2}}$;

\item If $n$ is odd, then $\sigma_{4}\leq\frac{1}{8}n^{2}\left(  n-1\right)
\left(  n-3\right)  $, with equality if and only if $P_{S}(x)=x(x^{2}%
+n)^{\frac{n-1}{2}}$.
\end{enumerate}
\end{proposition}

The proof of this proposition is based on the well known Maclaurin's inequality.

\begin{lemma}
Let $a_{1},a_{2},...,a_{l}$ be positive real numbers, and for $k=1,2,...,l$
define the averages $S_{k}$ as follows:
\[
S_{k}=\frac{s_k}{\binom{l}{k}}%
\]
Then
\[
S_{1}\geq\sqrt{S_{2}}\geq\sqrt[3]{S_{3}}\geq\cdots\geq\sqrt[l]{S_{l}}%
\]
with equality if and only if $a_{1}=\cdots=a_{l}$.
\end{lemma}

\begin{pop1}
The first assertion follows from equality (\ref{eq01}) and the fact that all
of the principal minors of order $2$ are equal to $1$.

To prove the second and third assertions, let $m$ be the integer part of $\frac{n}{2}$ and
$\pm i\lambda_{1},\ldots,\pm i\lambda_{m}$ the nonzero eigenvalues of $S$.

As we have seen above
\[
P_{S}(x)=%
\begin{cases}
(x^{2}+\lambda_{1}^{2})(x^{2}+\lambda_{2}^{2})\cdots(x^{2}+{\lambda_{m}^{2}%
}) & \text{ if }n\text{ is even }\\
x(x^{2}+\lambda_{1}^{2})(x^{2}+\lambda_{2}^{2})\cdots(x^{2}+{\lambda_{m}^{2}%
}) & \text{if }n\text{ is odd}%
\end{cases}
\]
By developing these products and identifying the coefficients of the obtained
polynomials with those of $x^{n}+\sigma_{2}x^{n-2}+\sigma_{4}x^{n-4}%
+\cdots+\sigma_{n-2}x^{2}+\sigma_{n}$, we get $\sigma_{2}=\sum_{1\leq i\leq
m}{{\lambda_{i}}^{2}}$ and $\sigma_{4}=\sum_{1\leq i<j\leq m}{{\lambda_{i}%
}^{2}{\lambda_{j}}^{2}}$.

Now, applying Maclaurin's inequality to $\lambda_{1}^{2},\lambda_{2}%
^{2},\ldots,\lambda_{m}^{2}$, for $k=2$, we obtain
\begin{equation}
\frac{(\sum_{1\leq i<j\leq m}{{\lambda_{i}}^{2}{\lambda_{j}}^{2})}}{\binom
{m}{2}}\leq\left(  \frac{1}{m}\underset{1\leq i\leq m}{\sum}\lambda_{i}%
^{2}\right)  ^{2} \label{eq1}%
\end{equation}

It follows that $\sigma_{4}\leq\frac{m-1}{2m}\sigma_{2}^{2}$ and hence, by
assertion 1, $\sigma_{4}\leq\frac{(m-1)n^{2}(n-1)^{2}}{8m}$

We conclude that

\begin{equation}
\sigma_{4}\leq%
\begin{cases}
\frac{1}{8}n(n-1)^{2}(n-2) & \text{if }n\text{ is even}\\
\frac{1}{8}n^{2}\left(  n-1\right)  \left(  n-3\right)  & \text{if }n\text{ is
odd}%
\end{cases}
\label{eq2}%
\end{equation}

Moreover, equality holds in \ref{eq1} and hence in \ref{eq2}, if and only if
$\lambda_{1}^{2}=\lambda_{2}^{2}=\ldots=\lambda_{m}^{2}=\frac{\sigma_{2}}{2}$.

It follows that if $n$ is even, then $\sigma_{4}=\frac{1}{8}n(n-1)^{2}(n-2)$
if and only if $P_{S}(x)=(x^{2}+(n-1))^{\frac{n}{2}}$ and if $n$ is odd, then $\sigma
_{4}=\frac{1}{8}n(n-1)^{2}(n-2)$ if and only if $P_{S}(x)=x(x^{2}+n)^{\frac{n-1}{2}}$.
\end{pop1}

Theorem \ref{main1} is a direct consequence of Proposition \ref{diamonds} and
the following proposition.
\begin{proposition}
\label{lemseidel} Let $S$ be a skew-symmetric Seidel matrix of order $n$. The
following assertions hold

\begin{enumerate}
\item If $n\equiv0\pmod{4}  $, then $P_{S}(x)=(x^{2}+(n-1))^{\frac
{n}{2}}$ if and only if $S$ is a skew-conference matrix.

\item If $n\equiv3\pmod{4}  $, then $P_{S}(x)=x(x^{2}+n)^{\frac
{n-1}{2}}$ if and only if $S$ is obtained by deleting a row and column from a
skew-conference matrix of order $n+1$.
\end{enumerate}
\end{proposition}

The proof of this proposition is contained implicitly in \cite{greaves17}. It is based on the following results.

\begin{theorem}
\label{cauchy}\textbf{(Cauchy Interlace Theorem)} Let $A$ be a Hermitian
matrix of order $n$ with eigenvalues $\lambda_{1}\geq\lambda_{2}\geq\cdots
\geq\lambda_{n-1}\geq\lambda_{n}$. Let $B$ be a principal submatrix of $A$ of
order $n-1$ with eigenvlaues $\mu_{1}\geq\mu_{2}\geq\cdots\geq\mu_{n-1}$.Then
$\lambda_{1}\geq\mu_{1}\geq\lambda_{2}\geq\cdots\geq\lambda_{n-1}\geq\mu
_{n-1}\geq\lambda_{n}$.
\end{theorem}

\begin{lemma}
\cite{greaves17}\label{skewconf} Let $S$ be a skew-symmetric Seidel matrix of
order $n\equiv3\pmod{4}  $ with characteristic polynomial
$P_{S}(x)=x(x^{2}+n)^{\frac{n-1}{2}}$. Then there exists a vector $\mathbf{v}$
with entries from $\left\{  -1,1\right\}  $ such that the characteristic
polynomial of the matrix $%
\begin{pmatrix}
S & \mathbf{v}\\
-\mathbf{v}^{T} & 0
\end{pmatrix}
$ is $(x^{2}+n)^{\frac{n+1}{2}}$.
\end{lemma}

\begin{pop2} 
Observe that $iS$ is a Hermitian matrix and
so is diagonalizable with real eigenvalues.

Assume that $n\equiv0\pmod{4}  $ and $P_{S}(x)=(x^{2}+(n-1))^{\frac
{n}{2}}$. Then $P_{iS}(x)=(x^{2}-(n-1))^{\frac{n}{2}}$. It follows that the
minimal polynomial of $iS$ is $x^{2}-(n-1)$ and hence $S^{2}+(n-1)I_{n}=0$,
where $I_{n}$ is the $n\times n$ identity matrix. Thus, $S$ is a
skew-symmetric conference matrix. The converse is trivial.

To prove the second assertion, assume that $n\equiv3\pmod{4}  $ and
$P_{S}(x)=x(x^{2}+n)^{\frac{n-1}{2}}$. By Lemma \ref{skewconf}, there exists a
vector $\mathbf{v}$ with entries from $\left\{  -1,1\right\}  $, such that the
characteristic polynomial of the matrix $\widehat{S}=%
\begin{pmatrix}
S & \mathbf{v}\\
-\mathbf{v}^{T} & 0
\end{pmatrix}
$ is $(x^{2}+n)^{\frac{n+1}{2}}$. By the first assertion, $\widehat{S}$ is a
skew-conference matrix. Conversely, suppose that $S$ is obtained by deleting a
row and the corresponding column from a skew-conference matrix $\widehat{S}$
of order $n+1$. The eigenvalues of $i\widehat{S}$ are $\lambda_{1}=\lambda
_{2}=\cdots=\lambda_{(n+1)/2}=\sqrt{n}$, and $\lambda_{(n+3)/2}=\cdots=\lambda
_{n+1}=-\sqrt{n}$. By Theorem \ref{cauchy} and since $0$ is an eigenvalue of
$iS$, the eigenvalues of $iS$ are $\mu_{1}=\mu_{2}=\cdots=\mu_{(n-1)/2}%
=\sqrt{n}$, $\mu_{(n+1)/2}=0$, $\mu_{(n+3)/2}=\cdots=\mu_{n}=-\sqrt{n}$.
It follows that $P_{S}(x)=x(x^{2}+n)^{\frac{n-1}{2}}$.
\end{pop2}

\section{Partial solution to Frankl-F\"uredi problem}

In their work \cite{semeraro17}, Gunderson and Semeraro obtained the maximum
number possible of hyperedges in an $r$-uniform hypergraph $\mathcal{H}$ of order $n$, with
the property that every set of $r+1$ vertices contains at most $2$ hyperedges.
More precisely, for such hypergraph, they proved that $e(\mathcal{H}%
)\leq\frac{n}{r^{2}}\binom{n}{r-1}$, with equality if and only if 
every set of $(r-1)$ vertices occurs in exactly $n/r$ hyperedges. Remark
that Proposition \ref{sem1} corresponds to the case $r=4$.

In order to prove our second main result, we need the following combinatorial
lemma. Its proof is similar to that of \cite[Lemma~1]{reid89}.

\begin{lemma}
\label{min} Suppose that $s$ and $p$ are positive integers. Write $s=pk+h$,
for some integers $k$ and $h$, $0\leq h<p$.

Let $\Gamma =\left\{  (x_{1},\ldots,x_{p})\ :x_{1}\leq\cdots\leq x_{p}%
,\ x_{i}\in\mathbb{N},\ 1\leq i\leq p,\ \overset{p}{\underset{i=1}{\sum}}%
x_{i}=s\right\}  $ \newline Then
\[
\min\left\{  \underset{i=1}{\overset{p}{\sum}}x_{i}^{2}\ :(x_{1},\ldots
,x_{p})\in\Gamma\right\}  =h(k+1)^{2}+(p-h)k^{2}%
\]
Moreover, for $(x_{1},\ldots,x_{p})\in\Gamma$, we have $\underset
{i=1}{\overset{p}{\sum}}x_{i}^{2}=h(k+1)^{2}+(p-h)k^{2}$ if and only $x_{i}%
\in\left\{  k,k+1\right\}  $ for $i=1,\ldots,p$.
\end{lemma}

\begin{pop3}

Let $C_{1},\ldots,C_{\binom{n}{3}}$ be the family of all $3$-sets of vertices
and for each $i\leq\binom{n}{3}$, let $a_{i}$ be the number of hyperedges of
$\mathcal{H}$ containing $C_{i}$. By double counting, we have $\underset
{i=1}{\overset{\binom{n}{3}}{\sum}}a_{i}=4e(\mathcal{H})$.

Let us write $\,4e(\mathcal{H})=\binom{n}{3}k+h\,$ where $k$ and $h$ are integers and
$0\leq h<\binom{n}{3}$.

By Proposition \ref{sem1}, we have $\,e(\mathcal{H})\leq\frac{1}{96}%
n^{2}(n-1)(n-2)$.

Then 
\begin{eqnarray*}
k & \leq & \frac{4e(\mathcal{H})}{\binom{n}{3}} \\
  & \leq &\frac{\tfrac{4}{96}n^{2}(n-1)(n-2)}{\binom{n}{3}}\\
    & = & \frac{1}{4}n\,=\ t+\frac{3}{4}
\end{eqnarray*}

Hence $\,k\leq t$.

If $\,k\leq t-1\,$ then $\, 4e(\mathcal{H})\leq(\tfrac{n-3}{4}-1)\tbinom{n}{3}
+\tbinom{n}{3}$.

It follows that 
\begin{eqnarray*}
e(\mathcal{H}) & \leq & \tfrac{1}{96}n\left(  n-3\right)  \left(n-1\right)  \left(  n-2\right) \\
               &\leq & \tfrac{1}{96}n\left(  n-1\right)  \left(n-3\right)  (n+1)
\end{eqnarray*}

Assume now that $k=t$, then $h=4e(\mathcal{H})-\tbinom{n}{3}t$. Hence, by
applying Lemma \ref{min}, we get
\[
\underset{i=1}{\overset{\binom{n}{3}}{\sum}}a_{i}^{2}\geq h\left(  t+1\right)
^{2}+\left(  \binom{n}{3}-h\right)  t^{2}\nonumber
\]

By substitution, we get
\begin{equation}
\underset{i=1}{\overset{\binom{n}{3}}{\sum}}a_{i}^{2}\geq\left(  2n-2\right)
e(\mathcal{H})\allowbreak-\frac{1}{96}n\left(  n-1\right)  \left(  n-2\right)
\left(  n-3\right)  \left(  n+1\right)  \label{eqyy}%
\end{equation}

We consider the set
\[
\mathcal{F=}\left\{  (A,B)\mid\left\vert A\right\vert =\left\vert B\right\vert
=4\text{, }\left\vert A\cap B\right\vert =3\text{, }A\in\mathcal{H}\text{ and
}B\notin\mathcal{H}\right\}
\]

This set is a disjoint union of the following sets
\[
\mathcal{F}_{i}=\left\{  (A,B)\mid\left\vert A\right\vert =\left\vert
B\right\vert =4\text{, }A\cap B=C_{i}\text{, }A\in\mathcal{H}\text{ and
}B\notin\mathcal{H}\right\}
\]

Using the proof of in \cite[Proposition~14]{semeraro17}, we get
\[
\left\vert \mathcal{F}\right\vert \geq3(n-4)e(\mathcal{H})
\]

and%
\[
\underset{i=1}{\overset{\binom{n}{3}}{\sum}}\left\vert \mathcal{F}%
_{i}\right\vert =4(n-3)e(\mathcal{H)-}\underset{i=1}{\overset{\binom{n}{3}%
}{\sum}}a_{i}^{2}%
\]

It follows that
\[
3(n-4)e(\mathcal{H})\leq4(n-3)e(\mathcal{H)-}\underset{i=1}{\overset{\binom
{n}{3}}{\sum}}a_{i}^{2}%
\]

Hence, by inequality \ref{eqyy}, we get
\begin{align*}
3(n-4)e(\mathcal{H})  &  \leq4(n-3)e(\mathcal{H)-}\underset{i=1}%
{\overset{\binom{n}{3}}{\sum}}a_{i}^{2}\\
&  \leq\left(  2n-10\right)  e(\mathcal{H})+\frac{1}{96}n\left(  n-1\right)
\left(  n-2\right)  \left(  n-3\right)  \left(  n+1\right)
\end{align*}

Then%
\[
e(\mathcal{H})\leq\frac{1}{96}n\left(  n-1\right)  \left(  n-3\right)  \left(
n+1\right)
\]

\end{pop3}

\section{Concluding remarks}

For an integer $t\geq1$, a $t-(n,k,\lambda)$ \emph{design} is an ordered pair
$\mathcal{D=}(V,\mathcal{B})$ where $V$ is a set of size $n$ and $\mathcal{B}$
is a collection of $k$-subsets of $V$ called \emph{blocks}, such that every
$t$-subset of $V$ is contained in exactly $\lambda$ blocks.

Recall the following well-known result about designs.

\begin{theorem}
\label{basic design}Let $\mathcal{D}=(V,\mathcal{B})$ be a $t-(n,k,\lambda)$
design. Then

\begin{enumerate}
\item $\left\vert \mathcal{B}\right\vert =\lambda\dfrac{\binom{n}{t}}%
{\binom{k}{t}}$.

%\item Every element of $V$ lies in exactly $\lambda\dfrac{\binom{n-1}{t-1}%
%}{\binom{k-1}{t-1}}$ blocks.

\item For $s=1,\cdots,t$, every $s$-subset of $V$ is contained in exactly $\lambda\dfrac
{\binom{n-s}{t-s}}{\binom{k-s}{t-s}}$ blocks.
\end{enumerate}
\end{theorem}

Semeraro and Gunderson 
\cite{semeraro17} raised the following question.  
For which natural numbers $n\equiv0\pmod{4}$ does there exist a $3-(n,4,\frac{n}{4})$ design 
with the property that every set of $5$ vertices contains either $0$ or $2$ hyperedges? 
We call such design an \emph{$FF_4$-design}. Following the definition of a design, the equality in Proposition \ref{sem1}
holds if and only if $\mathcal{H}$ is an $FF_4$-design.

\begin{remark}
    By Theorem \ref{main1}, an $FF_4$-design of order $n\equiv0\pmod{4}$ can be 
    obtained via Baber's construction from a skew-conference matrix of the same order. 
\end{remark}

Let $n\equiv3\pmod{4}  $ and assume that there is an $FF_{4}
$-design $\mathcal{H}$ with $n+1$ vertices. Let $x$ be an
arbitrary vertex of $\mathcal{H}$. By using assertion 2 of Theorem
\ref{basic design}, for $s=1$, it is easy to see that the $FF_{4}$-hypergraph
$\mathcal{H}-\left\{  x\right\}  $ has $n$ vertices and $\frac{1}{96}n\left(
n-3\right)  \left(  n-1\right)  \left(  n+1\right)$ hyperedges. 
Then, the bound in Proposition \ref{main2} is reached. The upper bound in
other cases seems to be difficult to find.

\begin{remark}
 By applying the second assertion of Theorem \ref{basic design}, and using
 the inclusion-exclusion principle, we obtain the following
\begin{enumerate}
    \item Let $n\equiv 2\pmod{4}$, by removing two vertices from an $FF_{4}$-design of order $n+2$, 
    we obtain an $FF_4$-hypergraph with $n$ vertices and $\frac
{1}{96}n(n-3)(n+2)(n-2)  $ hyperedges.
    \item  Let $n\equiv1\pmod{4}$, by removing two vertices from an $FF_{4}$-design of order $n+3$, 
    we obtain an $FF_4$-hypergraph with $n$ vertices and $\frac
{1}{96}(n-1)(n-2)(n-3)(n+3)$ hyperedges.
\end{enumerate}  
\end{remark}

Based on the previous remark, we may state the following conjecture.

\begin{conjecture}
Let $\mathcal{H}$ be an $FF_{4}$-hypergraph with $n$ vertices
\begin{enumerate}
\item if $n\equiv2\pmod{4}  $ then $\mathcal{H}$ has at most $\frac
{1}{96}n\left(  n-3\right)  \left(  n+2\right)  \left(  n-2\right)  $ hyperedges.
\item if $n\equiv1\pmod{4}  $ then $\mathcal{H}$ has at most $\frac
{1}{96}\left(  n-1\right)  \left(  n-2\right)  \left(  n-3\right)  \left(
n+3\right)  $ hyperedges.
\end{enumerate}
\end{conjecture}

\bibliography{bibpaper}
\nocite{*}

\end{document}